\documentclass[hidelinks,onefignum,onetabnum]{siamart220329}

\usepackage{todonotes}
\usepackage{algorithmic}
\usepackage{algorithm}
\usepackage{amsmath}
\usepackage{amssymb}
\usepackage{graphicx}
\usepackage{listings}

\usepackage{wrapfig}
\usepackage{tikz}

\newcommand{\R}{{\mathbb R}}

\newcommand{\ass}{\,\mbox{:=}\,}

\usepackage{graphicx}

\begin{document}

\title{Differential Inversion of the Implicit Euler Method: 
Symbolic Analysis}

\author{Uwe Naumann
\thanks{Software and Tools for Computational Engineering,
               RWTH Aachen University,
	       52056 Aachen, Germany,
	      {\tt naumann@stce.rwth-aachen.de}}
	      }

\maketitle

\begin{abstract}
The implicit Euler method integrates systems of ordinary differential equations 
$$\frac{d x}{d t}=G(t,x(t))$$ with differentiable right-hand side 
	$G : \R \times \R^n \rightarrow \R^n$ from an initial state 
	$x=x(0) \in \R^n$ to a target time $t \in \R$ as $x(t)=E(t,m,x)$ using an
	equidistant discretization of the time interval $[0,t]$ yielding $m>0$ time steps. We present a method for efficiently computing the product of its inverse Jacobian
	$$
	(E')^{-1} \equiv
	\left (\frac{d E}{d x}\right )^{-1} \in \R^{n \times n}
	$$
	with a given vector $v \in \R^n.$
	We show that the differential inverse $(E')^{-1} \cdot v$ 
	can be evaluated for given $v \in \R^n$ with 
a computational cost of $\mathcal{O}(m \cdot n^2)$ as opposed to the standard
$\mathcal{O}(m \cdot n^3)$ or, naively, even $\mathcal{O}(m \cdot n^4).$
The theoretical results are supported by actual run times. 
A reference implementation is provided.
\end{abstract}

\section{Introduction} \label{sec:0}

The following is inspired by \cite{Naumann2024AME}. Therein a given 
implementation of a system of nonlinear equations
\begin{equation} \label{eqn:0}
y=f(x)=0\; , \quad  
f: \R^n \rightarrow \R^n
\end{equation}
as a differentiable program with Jacobian
\begin{equation} \label{eqn:jac}
	f' \equiv \frac{d f}{d x} \in \R^{n \times n}
\end{equation}
is assumed to be composed of differentiable elemental
subprograms
\begin{equation} \label{eqn:elem}
f_i : \R^n \rightarrow \R^n : x_i\ass f_i(x_{i-1}) 
\end{equation}
for $i=1,\ldots,m.$
Hence,
\begin{equation} \label{eqn:1}
x_m \ass f(x_0)=f_m(\ldots f_1(x_0) \ldots ) \; ,
\end{equation}
where $x_0=x$ and $y=x_m.$ 
Application of the chain rule of differentiation to
Equation~(\ref{eqn:1}) yields
$$
f'=f'_m \cdot \ldots \cdot f'_1 \; .
$$ 
We write ``$\ass$'' for imperative assignment. Mathematical 
equality is denoted as ``$=$'' and ``$\equiv$'' is to be read in the sense of
``is defined as.'' Approximate equality is denoted by ``$\approx$.''
The computational cost of each $f_i$ is expected to be
at least of order $n^2$ ($\mathcal{O}(n^2)$), which complies with most practically relevant
scenarios. 
Consequently, the computational cost of evaluating $f$ is $\mathcal{O}(m \cdot n^2).$

Without loss of generality (w.l.o.g.), all $f_i$ are assumed 
to be nonlinear, that is, $f'_i=f'_i(x_{i-1})$. 
Newton's method solves Equation~(\ref{eqn:0}) by driving the residual
$y=f(x)$ iteratively towards zero as
\begin{equation} \label{eqn:newton}
	x_{j+1}=x_j - \left ( f'(x_j) \right )^{-1} \cdot f(x_j) \; , \quad j=0,\ldots,p-1 \; .
\end{equation}
Convergence after $p \geq 0$ iterations is defined, for example, as the norm
of the residual $f(x_{j+1})$ falling below a given threshold $0 \leq \delta \ll 1.$ Applicability of Newton's method depends on a number of conditions, 
which we assume to be satisfied, see, for example, \cite{Deuflhard2004NMf,Kelley2003SNE}. Most importantly, a starting
value $x_0$ within proximity of the solution is typically required. 
Note the use of distinct indexes $i$ and $j$ in Equations~(\ref{eqn:elem}) and
(\ref{eqn:newton}), respectively. They will be combined as part of a unified 
notation in the following.

The Jacobian of $f$ can be computed with machine accuracy
by Algorithmic 
Differentiation (AD) \cite{Griewank2008EDP}. The vector tangent mode of AD (also: tangent AD) yields 
\begin{equation}
	(y,\dot{Y})=\dot{f}(x,\dot{X}) \equiv (f(x), f'(x) \cdot \dot{X}) 
\end{equation}
for $\dot{X} \in \R^{n \times \dot{n}}$ with $1 \leq \dot{n} \leq n.$
It enables the accumulation of (a dense) $f'$ with a computational cost of 
$\mathcal{O}(m \cdot n^3)$ by setting $\dot{X}$ equal to the identity
$I_n \in \R^{n \times n}.$ Potential sparsity of $f'$ can and should be exploited \cite{Gebremedhin2005WCI}. Corresponding numerical approximations can be 
obtained with a cost of the same order by finite differences. 

Adjoint AD in vector mode yields 
\begin{equation} \label{eqn:aad}
	(y,\bar{X})=\bar{f}(x,\bar{Y}) \equiv (f(x), (f'(x))^T \cdot \bar{Y}) 
\end{equation}
for $\bar{Y} \in \R^{m \times \bar{m}},$ $1 \leq \bar{m} \leq m.$
Again, the computational cost of accumulating (a dense) $f'$ is equal to
$\mathcal{O}(m \cdot n^3)$ as $\bar{Y}=I_n.$ 
Adjoint AD is of particular interest in the context of large-scale 
gradient-based numerical methods. In such cases $\bar{Y}=1$ yields cheap
gradients with constant relative (with respect to the cost of evaluating $f : \R^n \rightarrow \R$) 
computational cost.
Tangent AD is typically preferred for square Jacobians to avoid the overhead
induced by data flow reversal in adjoint AD \cite{Griewank1992ALG,Naumann2008DRi}. Details are beyond 
the scope of this paper. Refer to {\tt www.autodiff.org} for a 
comprehensive bibliography on AD.

\begin{lemma} \label{lem:1}
	$$
	\left (f'\right )^{-1} = \left ( f'_m \cdot \ldots \cdot f'_1 \right )^{-1} =
	(f'_1)^{-1} \cdot \ldots \cdot (f'_m)^{-1} 
	$$
\end{lemma}
\begin{proof}
	\begin{align*}
		(f')^{-1} \cdot f'&=
	\left ( f'_m \cdot \ldots \cdot f'_1 \right )^{-1} 
	\cdot (f'_m \cdot \ldots \cdot f'_1) \\
		&=
		(f'_1)^{-1} \cdot \ldots \cdot \underset{=I_n}{\underbrace{((f'_m)^{-1} \cdot f'_m)}} \cdot \ldots \cdot f'_1 =I_n \; .
	\end{align*}
\end{proof}
With Lemma~\ref{lem:1}, the Newton step 
$\Delta x_j=-(f'(x_j))^{-1} \cdot f(x_j)$ becomes equal to 
\begin{equation} \label{eqn:3}
        \begin{split}
                \Delta x_j&=
		-(f'_m(x_{m-1,j}) \cdot \ldots \cdot f'_1(x_{0,j}))^{-1} \cdot f(x_{0,j}) \\
		&=-f'_1(x_{0,j})^{-1} \cdot \ldots \cdot f'_m(x_{m-1,j})^{-1} \cdot f(x_{0,j}) \; ,
        \end{split}
\end{equation}
where $x_j=x_{0,j}.$
The Jacobians $f'(x_j)$ are assumed to be invertible at all iterates. Invertibility of all $f'_i=f'_i(x_{i-1,j})$ 
follows immediately.
Evaluation of Equation~(\ref{eqn:3}) as
$$
                \Delta x_j=
		-f'_1(x_{0,j})^{-1} \cdot (\ldots \cdot (f'_m(x_{m-1,j})^{-1} \cdot f(x_{0,j})) \ldots )
		$$
amounts to the solution of $2m$ linear systems. 
This method is matrix-free in the sense that a potentially dense $f'$ is not 
computed explicitly. Potential sparsity of the $f'_i$ can be exploited
in the context of {\em structural analysis} as discussed in 
\cite{Naumann2024AME}, where, for example, a reduction of the computational 
cost of differential inversion from $\mathcal{O}(m \cdot n^3)$ to 
$\mathcal{O}(m \cdot n^2)$ is reported for tridiagonal $f'_i.$

This paper's focus lies on {\em symbolic analysis}. Its results are applicable
to dense (as well as to sparse) $f'_i.$ Similar reductions in computational cost can be achieved. 
Section~\ref{sec:1} recalls the implicit Euler and Newton methods as essential prerequisites. 
Section~\ref{sec:3} represents the heart of this contribution. Starting with 
a naive (black-box) approach to differential inversion of any differentially 
invertible program, and of the implicit Euler method in particular, we discuss
two improvements yielding a reduction of the computational cost by 
$\mathcal{O}(n^2).$ Reference implementations for all three approaches are
presented in Section~\ref{sec:4} together with run time measurements 
in Section~\ref{sec:5}. Conclusions are drawn in Section~\ref{sec:6}.

\section{Prerequisites} \label{sec:1}
The implicit (also: backward) Euler method 
\begin{equation} \label{eqn:euler}
	E : \R \times \mathbb{N} \times \R^n \rightarrow \R^n : y=E(t,m,x)
	\end{equation}
integrates systems of ordinary differential equations 
\begin{equation} \label{ode}
\frac{d x}{d t}=G(t,x(t)) 
\end{equation}
with differentiable right-hand side 
	$G : \R \times \R^n \rightarrow \R^n$ from a given initial state 
$x=x(0) \in \R^n$ to a given target time $t=\check{t} \in \R$ as $y=x(\check{t})=E(\check{t},\check{m},x(0))$ using an
	equidistant discretization of the time interval $[0,\check{t}]$ with a given number $m=\check{m}>0$ time steps. 
In the following, $E$ is required to be {\em differentially invertible}
at $\check{t},\check{m},x(0)$, that is, it needs to be differentiable and its 
Jacobian must be invertible.
The corresponding {\em differential inverse}
	$$
	\left (\frac{d E}{d x}(\check{t},\check{m},x(0)) \right )^{-1} \cdot v \in \R^{n}
	$$
	is, for example, required in the context of the inverse problem, aiming to 
	estimate the initial state $x(0)$ for given observations of $x(\check{t}).$ 
	Feasibility of this inversion depends on a number of conditions, 
	which are assumed to be satisfied, see, for example, \cite{Chavent2010NLS}. Generalization for scenarios that 
	require regularization is the subject of ongoing investigations.

\subsection{Integration (Outer Iteration)}

The implicit Euler method approximates the time derivative in 
Equation~(\ref{ode}) with a backward finite difference obtained by
truncation of the Taylor expansion 
$$
x(t)=x(t-\Delta t) +
\frac{d x}{d t}\cdot \Delta t+ {\mathcal O}(\Delta t^2) 
$$
of $x$ at $t-\Delta t$ after the first-order term, yielding the linearization
$$
\frac{d x}{d t} \approx \frac{x(t)-x(t-\Delta t)}{\Delta t} \; .
$$
Equation~(\ref{ode}) is replaced by its discrete equivalent
$$
\frac{x(t)-x(t-\Delta t)}{\Delta t}=G(t,x(t)) \; . 
$$
Time steps of size $0 < \Delta t \leq t$ yield a sequence of iterates
$x_{i},$ $i=1,\ldots m,$ such that
\begin{equation} \label{dode}
\frac{x_{i}-x_{i-1}}{\Delta t} = G(i \cdot \Delta t,x_{i})
\end{equation}
and where $x_{i}=x(i\cdot \Delta t)$ implying $x_{m}=x(m \cdot \Delta t)=x(t).$ Uniform time stepping yields $\Delta t=\frac{t}{m}.$
Convergence of the implicit Euler method 
depends on a number of conditions, 
which are assumed to be satisfied. See, for example, \cite{Butcher2016NMf} for details.

\subsection{Root Finding (Inner Iteration)}

Equation~(\ref{dode}) implies the system of nonlinear equations 
\begin{equation} \label{res}
F(x_{i},x_{i-1},\Delta t)=x_{i}-x_{i-1}-\Delta t \cdot G(i \cdot \Delta t,x_{i})=0 \; .
\end{equation}
It needs to be solved $m$ times to obtain the solution $x(t).$
Linearization of $F$ at $x_{i}$ yields
\begin{align*}
0=F(x_{j+1,i},x_{i-1},\Delta t) &= F(x_{j,i}+\Delta x_{j,i},x_{i-1},\Delta t) \\
&=
F(x_{j,i},x_{i-1},\Delta t)+
\frac{d F}{d x_{i}}(x_{j,i},x_{i-1},\Delta t) \cdot \Delta x_{j,i} \; .
\end{align*}
Setting
$F' \equiv \frac{d F}{d x_{i}},$
the solution of the resulting linear system
$$
F'(x_{j,i},x_{i-1},\Delta t) \cdot 
\Delta x_{j,i} = - F(x_{j,i},x_{i-1},\Delta t)
$$
is followed by updates
$$
x_{j+1,i}= x_{j,i}+\Delta x_{j,i}, \quad j=0,\ldots,p \; ,
$$
where, for the purpose of cost analysis, convergence is assumed to be reached
after (at most) $p$ Newton steps for all $m$ iterations of the implicit Euler method. W.l.o.g., we use $x_{0,i}=x_{i-1}.$ 

Differentiation of Equation~(\ref{res}) with respect to $x_{i}$ 
yields
$$
F'(x_{j,i},x_{i-1},\Delta t)
=I_n-x_{i-1}-\Delta t \cdot G'(i \cdot \Delta t,x_{j,i})=0 
$$
which is solved by Newton's method 
$$N : \R^n \times \R^n \times \R \rightarrow \R^n : \quad x_{i}=N(x_{0,i},x_{i-1},\Delta t) $$ 
for 
given $G' \equiv \frac{d G}{d x_{i}}.$

Implicit Euler integration
amounts to the solution of $m$ systems of nonlinear equations at the expense
of $\mathcal{O}(p \cdot n^3)$ each. Both accumulation and factorization of
(the dense)
$F'$ induce a computational cost of $\mathcal{O}(n^3)$ assuming evaluation of
$F$ at a cost of $\mathcal{O}(n^2).$ Potential sparsity of $F'$ can and should 
be 
exploited.
The total cost of the implicit Euler method can, hence, be estimated as 
$\mathcal{O}(m \cdot p \cdot n^3)$.

\section{Differential Inversion} \label{sec:3}

We present three alternatives for differential inversion of the implicit Euler method.

\subsection{Black-box Approach}

The implicit Euler method 
\begin{equation} \label{implicit_euler_scheme}
x_{i}=N(x_{0,i},x_{i-1},\Delta t)\; , \quad i=1,\ldots,m \; ,
\end{equation}
can be differentiated naively as a black box using, for example, tangent AD,
at the expense of ${\mathcal O}(n)$ evaluations of $E,$ that is,
at ${\mathcal O}(m \cdot p \cdot n^4);$ see also Algorithm~\ref{alg3.1}.
\begin{algorithm} 
\begin{alignat*}{1}
\dot{X}_{0}&\ass I_n \\
(x_{m}, \dot{X}_{m})&\ass \dot{E}(\check{t},\check{m},x_{0},\dot{X}_{0}) \\
\frac{d x_{m}}{d x_{0}}&\ass \dot{X}_{m} \\
w&\ass \left (\frac{d x_{m}}{d x_{0}}\right )^{-1} \cdot v
\end{alignat*}
\caption{Naive Differential Inversion of the Implicit Euler Method} \label{alg3.1}
\end{algorithm}

Denoting $E'=\frac{d E}{d x},$
the resulting Jacobian
$$
E'=E'(\check{t},\check{m},x_{0}) =
\frac{d x_{m}}{d x_{0}} \in \R^{n \times n}
$$
is assumed to be invertible.
It becomes the system matrix of the linear system
$$
E' \cdot w = v
$$
whose solution yields $$w = \left (E' \right )^{-1} \cdot v$$
with an additional, yet insignificant, cost of ${\mathcal O}(n^3).$ 
The computational effort is clearly dominated by the differentiation of
the implicit Euler method.

\subsection{Partially Symbolic Approach}

\begin{lemma} \label{lem:2}
	\begin{equation} \label{eqn:alg3.2}
		E'(\check{t},\check{m},x_{0}) =
	\left ( \frac{d F}{d x_m}\right )^{-1} \cdot  \ldots  \cdot 
	\left ( \frac{d F}{d x_{2}}\right )^{-1} \cdot 
	\left ( \frac{d F}{d x_{1}}\right )^{-1} \; .
	\end{equation}
\end{lemma}

\begin{proof}
Equation~(\ref{res}), that is
$F(x_{i},x_{i-1},\Delta t)=0,$ defines $x_i$ implicitly as a function of 
$x_{i-1}.$
Differentiation with respect to $x_{i-1}$ yields
$$
\frac{d F}{d x_{i-1}}
	=
\frac{d F}{d x_{i}} \cdot
	\frac{d x_{i}}{d x_{i-1}} + \frac{\partial F}{\partial x_{i-1}}
=0
$$
and, hence,
$$
	\frac{d x_{i}}{d x_{i-1}}=-\left (\frac{d F}{d x_{i}} \right )^{-1} \cdot \frac{\partial F}{\partial x_{i-1}}
=-\left (\frac{d F}{d x_{i}} \right )^{-1} \cdot -I_n 
=\left (\frac{d F}{d x_{i}} \right )^{-1} \; .
$$
According to the chain rule of differentiation,
\begin{align*}
E'(\check{t},\check{m},x_{0}) =
	\frac{d x_{m}}{d x_{0}} &=
	\frac{d x_{m}}{d x_{m-1}} \cdot  \ldots  \cdot 
	\frac{d x_{2}}{d x_{1}} \cdot 
	\frac{d x_{1}}{d x_{0}} \ldots \\ &=
	\left ( \frac{d F}{d x_m}\right )^{-1} \cdot  \ldots  \cdot 
	\left ( \frac{d F}{d x_{2}}\right )^{-1} \cdot 
	\left ( \frac{d F}{d x_{1}}\right )^{-1} \ldots \; ,
\end{align*}
which completes the proof.
\end{proof}
\begin{algorithm}
\begin{alignat*}{2}
\dot{X}_{0}&\ass I_n \\
\text{for}~&i=1,\ldots,m: \\
&&\frac{d x_{i}}{d x_{i-1}}&\ass \left (\frac{d F}{d x_{i}}(x_{i},x_{i-1},\Delta t) \right )^{-1} = \left (L_i \cdot U_i \right )^{-1} \\
&&\dot{X}_{i} &\ass \frac{d x_{i}}{d x_{i-1}} \cdot \dot{X}_{i-1} =\left (L_i \cdot U_i \right )^{-1} \cdot \dot{X}_{i-1} \\
\frac{d x_{m}}{d x_{0}}&=\dot{X}_{m} \\
	w&\ass \left (\frac{d x_{m}}{d x_{0}}\right )^{-1} \cdot v
\end{alignat*}
\caption{Partially Symbolic Differential Inversion of the Implicit Euler Method} \label{alg3.2}
\end{algorithm}
The resulting Algorithm~\ref{alg3.2} formalizes the corresponding augmentation of the implicit Euler method. AD of the Newton algorithm is avoided, thus 
reducing the computational cost by a factor of $\mathcal{O}(p \cdot n);$ 
see also \cite{Gilbert1992Ada,Naumann2015NLS}. Associativity of matrix multiplication ensures feasibility of
bracketing Equation~(\ref{eqn:alg3.2}) from the right. The total computational cost of differential inversion becomes $\mathcal{O}(m \cdot n^3)$ due to 
repeated inversion (w.l.o.g., using $LU$ decomposition) of the $\frac{d F}{d x_{i}}$ in addition to the underlying
implicit Euler scheme.

\subsection{Fully Symbolic}
A small additional step yields the fully symbolic method.
\begin{theorem}
\begin{equation} \label{Einv}
	E'(t,m,x_{0})^{-1} \cdot v=\frac{d F}{d x_{i}}(x_{1},x_{0},\Delta t) \cdot \ldots \cdot 
\frac{d F}{d x_{i}}(x_{m},x_{m-1},\Delta t) \cdot v \; ,
\end{equation}
\end{theorem}
\begin{proof}
	This result follows immediately from Lemma~\ref{lem:2}.
	Equation~(\ref{eqn:alg3.2}) implies
\begin{align*}
	E'(t,m,x_{0})^{-1} 
	&= \left (\frac{d x_{m}}{d x_{m-1}} \cdot \ldots
	\cdot \frac{d x_{1}}{d x_{0}} \right )^{-1} \\
	&= \left ( \frac{d x_{1}}{d x_{0}} \right )^{-1} \cdot \ldots
\cdot \left ( \frac{d x_{m}}{d x_{m-1}} \right )^{-1} \\
	&= 
	\frac{d F}{d x_{1}}(x_1,x_0,\Delta t) \cdot \ldots \cdot 
	\frac{d F}{d x_{1}}(x_m,x_{m-1},\Delta t)
\end{align*}
and, hence, the claim of the theorem.
\end{proof}
Matrix-vector products involving the $\frac{d F}{d x_{i}}(x_i,x_{i-1},\Delta t)$
are performed for $i=m,\ldots,1$ as formalized in 
Algorithm~\ref{alg3.3}. 
Storage of $(n \times n)$-Jacobians at the end of each 
of the $m$
implicit Euler steps on a stack allows for fast differential inversion  
at the expense of the additional memory requirement of ${\mathcal O}(m \cdot n^2).$
The computational cost becomes equal to $\mathcal{O}(m \cdot n^2)$ as explicit matrix inversion
can be avoided entirely.

Note the analogy with adjoint AD. According to Equation~(\ref{eqn:aad}), 
the latter evaluates 
for differentiable programs as in Equation~(\ref{eqn:1}) 
\begin{align*}
	(f')^T \cdot \bar{Y} &=
	(f'_m \cdot \ldots \cdot f'_1)^T \cdot \bar{Y} \\
	&=
	(f'_1)^T \cdot \left ( \ldots \cdot \left ( (f'_m)^T \cdot \bar{Y} \right ) \ldots \right ) 
\end{align*}
with a computational cost of $\mathcal{O}(\bar{m} \cdot n^2)$ for 
$\bar{Y} \in \R^{n \times \bar{m}}.$ Bracketing from the left would result
in a cost of $\mathcal{O}(n^3).$ Most prominently, $\bar{m}=1$ for gradients.

\begin{algorithm}
\begin{alignat*}{2}
\text{for}~i&=1,\ldots,m: \\
	&&\text{push}&\left (\frac{d F}{d x_{i}}(x_{i},x_{i-1},\Delta t) \right ) \\
	w\ass &\;v \\
\text{for}~i&=m,\ldots,1: \\
	&&\text{pop}&\left (\frac{d x_i}{d x_{i-1}} \right ) \\
	&&w&\ass \frac{d x_{i}}{d x_{i-1}} \cdot w \\
\end{alignat*}
\caption{Fully Symbolic Differential Inversion of the Implicit Euler Method} \label{alg3.3}
\end{algorithm}

\section{Reference Implementation} \label{sec:4}

Our reference implementation is based on the following instance of the popular 
Lotka-Volterra equations \cite{Berryman1992Toa}
\begin{equation} \label{eqn:lv}
	\begin{split}
		\frac{d x_0}{d t}&=1.1\cdot x_0-0.5\cdot x_0\cdot x_1 \\
\frac{d x_1}{d t}&=-0.75\cdot x_1+0.25\cdot x_0\cdot x_1
	\end{split}
\end{equation} 
modelling the instantaneous growth rates of two populations consisting of
prey ($x_0$) and predators ($x_1$). Starting from given population sizes,
we integrate to time $\check{t}=1$ using $\check{m}=10^3$ time steps. The
code is written in C++ with Eigen\footnote{https://eigen.tuxfamily.org} 
employed for linear algebra.

The right-hand side $G$ from Equation~({\ref{ode}) 
\begin{lstlisting}
template<typename T>
VT<T> G(const VT<T> &x) {
  VT<T> r;
  r(0)=1.1*x(0)-0.5*x(0)*x(1); // prey
  r(1)=-0.75*x(1)+0.25*x(0)*x(1); // predators
  return r;
}
\end{lstlisting}
uses statically sized base-type-generic vector
\begin{lstlisting}
template<typename T>
using VT=Eigen::Vector<T,n>;
\end{lstlisting}
and matrix 
\begin{lstlisting}
template<typename T>
using MT=Eigen::Matrix<T,n,n>;
\end{lstlisting}
types provided by the Eigen library for given global \lstinline{n=2}.
Templates facilitate instantiation with different base types \lstinline{T}.
An implementation of the corresponding Jacobian follows immediately.
\begin{lstlisting}
template<typename T>
MT<T> dGdx(const VT<T> &x) {
  MT<T> Gx;
  Gx(0,0)=1.1-0.5*x(1);
  Gx(0,1)=-0.5*x(0);
  Gx(1,0)=0.25*x(1);
  Gx(1,1)=-0.75+0.25*x(0);
  return Gx;
}
\end{lstlisting}
The residual $F$ from Equation~(\ref{res}) is implemented as
\begin{lstlisting}
template<typename T>
VT<T> F(const VT<T> &x, const VT<T> &x_prev) {
  return x-x_prev-G(x)/m;
}
\end{lstlisting}
where $x_i$ is represented by \lstinline{x} and 
$x_{i-1}$ by \lstinline{x_prev}. The global integer variable \lstinline{m} 
holds the value of $\check{m}.$ An implementation of the Jacobian of the 
residual with respect to $x_i$ follows immediately.
\begin{lstlisting}
template<typename T>
MT<T> dFdx(const VT<T> &x) {
  return MT<T>::Identity()-dGdx(x)/m;
}
\end{lstlisting}
It is used in Newton's method as follows.
\begin{lstlisting}
template<typename T>
VT<T> N(VT<T> x) {
  VT<T> x_prev=x,r=F(x,x_prev);
  do {
    x=x+dFdx(x).lu().solve(-r);
    r=F(x,x_prev);
  } while (r.norm()>1e-12);
  return x;
}
\end{lstlisting}
The implicit Euler method amounts to \lstinline{m} consecutive calls of
the above.
\begin{lstlisting}
template<typename T>
VT<T> E(VT<T> x) {
  for (int i=0;i<m;i++) x=N(x);
  return x;
}
\end{lstlisting}

\subsection{Implementation of Algorithm~\ref{alg3.1}}

Naive application of AD to the given implementation of $E$ 
is to be avoided due to suboptimal computational cost. Nevertheless,
we include 
an example based on the AD library dco/c++\footnote{https://nag.com/automatic-differentiation} \cite{lotz2011dco} for reference. Replication with other readily
available AD software for C++, for example, Adept \cite{Hogan2014FRM}, ADOL-C \cite{Griewank1996AAC}, CoDiPack \cite{sagebaum2019high} should be straightforward. Refer to {\tt http://www.autodiff.org} for a more complete list of AD software tools.

The function \lstinline{E_dEdx} returns a pair consisting of the 
solution for the initial value problem 
and its Jacobian at the initial state passed as the sole argument \lstinline{x}.
\begin{lstlisting}[numbers=left,numberstyle=\scriptsize]
std::pair<VT<double>,MT<double>> E_dEdx(VT<double> x) {
 Eigen::Vector<typename dco::gt1v<double,n>::type,n> x_t;
  for (int i=0;i<n;i++) {
    dco::value(x_t(i))=x(i);
    dco::derivative(x_t(i))[i]=1;
  }
  x_t=E(x_t);
  MT<double> E_x;
  for (int i=0;i<n;i++) {
    x(i)=dco::value(x_t(i));
    for (int j=0;j<n;j++)
      E_x(i,j)=dco::derivative(x_t(i))[j];
  }
  return std::make_pair(x,E_x);
}
\end{lstlisting}
dco/c++ provides the statically sized (\lstinline{n=2}) vector tangent type 
\lstinline{dco::gt1v<T,n>::type} over variable base type \lstinline{T} 
(equal to \lstinline{double} in this case); see line 2. Custom non-member 
functions allow
for read/write access to values (\lstinline{dco::value}; lines 4 and 10) 
and tangents (\lstinline{dco::derivative}; lines 5 and 12). 
Line 7 runs the overloaded implicit Euler method over variables of type
\lstinline{dco::gt1v<double,n>::type}. Appropriate instances of all functions involved are generated automatically by the compiler based on the given 
C++ templates. The result overwrites \lstinline{x_t} with the required
solution for the initial value problem 
and with its Jacobian. Both are used to evaluate
the differential inverse 
$$
E'(\check{t}, \check{m}, x_0)^{-1} \cdot  
E(\check{t}, \check{m}, x_0)
$$
as follows:
\begin{lstlisting}
VT<double> DifferentialInverse(const VT<double> &x) {
  std::pair<VT<double>,MT<double>> p=E_dEdx(x);
  return p.second.lu().solve(p.first);
}
\end{lstlisting}

\subsection{Implementation of Algorithm~\ref{alg3.2}}

Symbolic evaluation of 
$E'$ yields a modified version of 
\lstinline{E_dEdx}.
\begin{lstlisting}[numbers=left,numberstyle=\scriptsize]
std::pair<VT<double>,MT<double>> E_dEdx(VT<double> x) {
  MT<double> E_x=MT<double>::Identity();
  for (int i=0;i<m;i++) {
    x=N(x);
    E_x=dFdx(x).lu().solve(E_x);
  }
  return std::make_pair(x,E_x);
}
\end{lstlisting}
Differential inverses in the Cartesian basis directions (line 2) are 
propagated in line 5 alongside the implicit Euler steps evaluated in line 4. 
The other code remains unchanged.

\subsection{Implementation of Algorithm~\ref{alg3.3}}

The function \lstinline{E_dEdx} is no longer required by the implementation of
Algorithm~\ref{alg3.3}. In
\begin{lstlisting}[numbers=left,numberstyle=\scriptsize]
VT<double> DifferentialInverse(VT<double> x) {
  x=E(x);
  while (!tape.empty()) { x=tape.top()*x; tape.pop(); }
  return x;
}
\end{lstlisting}
the solution for the initial value problem is computed in line 2
followed by a sequence of matrix-vector products in line 3 with 
Jacobians of all implicit Euler steps stored on a stack
\begin{lstlisting}
std::stack<MT<double>> tape;
\end{lstlisting}
The Jacobians are pushed onto the {\em tape}\footnote{The term ``tape'' is motivated by the conceptual similarity of differential inversion with adjoint AD, where tapes are used for data flow reversal.} at the end of each implicit Euler step; see line 7 in the following.
\begin{lstlisting}[numbers=left,numberstyle=\scriptsize]
VT<double> N(VT<double> x) {
  VT<double> x_prev=x,r=F(x,x_prev);
  do {
    x=x+dFdx(x).lu().solve(-r);
    r=F(x,x_prev);
  } while (r.norm()>1e-12);
  tape.push(dFdx(x));
  return x;
}
\end{lstlisting}
The program
\begin{lstlisting}
int main(){
  std::cout << DifferentialInverse(VT<double>::Ones()) 
            << std::endl;
  return 0;
}
\end{lstlisting}
assumes initially unit population sizes resulting in the solution
$$
x(\check{t})=
\begin{pmatrix}
1.31161 \\
0.593445
\end{pmatrix} \; .
$$

\begin{figure}
\centering \includegraphics[width=.75\textwidth]{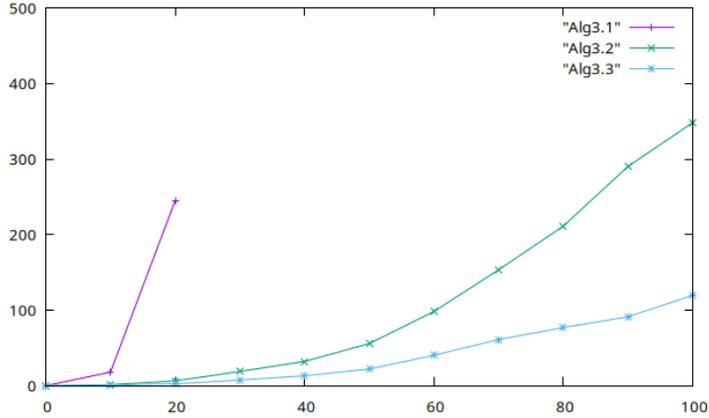}
	\caption{User run times (in $s$) for growing values of $n$ } \label{fig:A1}
\end{figure}

\begin{figure}
\centering \includegraphics[width=.75\textwidth]{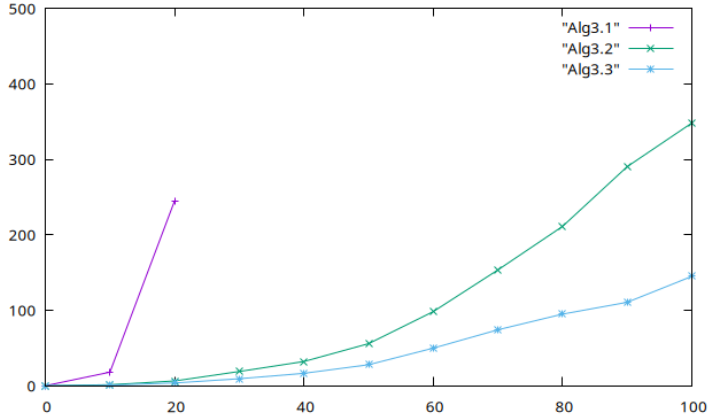}
	\caption{Elapsed run times (in $s$) for growing values of $n$} \label{fig:A2}
\end{figure}

\begin{figure}
\centering \includegraphics[width=.75\textwidth]{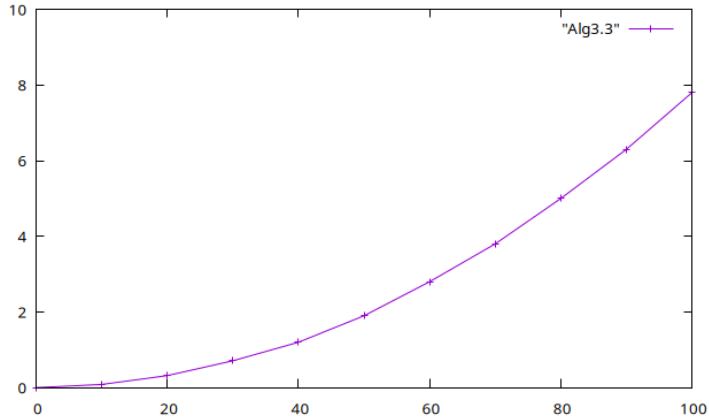}
	\caption{Evolution of resident set size (in $GB$) of Algorithm~\ref{alg3.3} for growing values of $n$} \label{fig:A3}
\end{figure}

\section{Experiments} \label{sec:5}

Our 
run time experiments aim to relate the theoretically obtained
computational complexities to an actual compute environment.\footnote{Intel Core I7, $16GB$ of RAM, GNU C++ compiler, Ubuntu Linux} Therefore, we use a 
generalized, scalable (in $n$) version of the Lotka-Volterra equations defined as 
$$
\frac{d x_k}{d t} = x_k \cdot f(x) \; ,
$$
where $f(x)=r+A \cdot x$ for given $r \in \R^n$ and 
$A \in \R^{n \times n};$ see \cite{Hofbauer1998EGa} for further details. For 
all three algorithms, we compare user and elapsed run times for ten\footnote{Better stability of the measured run times can thus be expected even for smaller problem instances.} 
differential inversions of $10^5$ implicit Euler steps for 
$n=0,10,\ldots,100$
in Figures~\ref{fig:A1} and \ref{fig:A2}, respectively. As predicted, Algorithm~\ref{alg3.1} becomes computationally expensive even for relatively small 
values of $n.$ 
Algorithm~\ref{alg3.3} ``beats'' Algorithm~\ref{alg3.2} by increasing factors.
The gap is smaller when considering elapsed run times due to additional system 
time to be devoted to handling storage and recovery of the Jacobians of the 
implicit Euler steps in Algorithm~\ref{alg3.3}.
The evolution of the corresponding stack size is shown in Figure~\ref{fig:A3}. 

\section{Conclusion} \label{sec:6}

Conceptually, differential inversion of a differentially invertible composite function
$$f(x)=f_m ( f_{m-1} ( \ldots f_1(x) \ldots ) : \R^n \rightarrow \R^n$$ 
involves two steps:
\begin{enumerate}
	\item accumulation of the Jacobian $f'=\frac{d f}{d x} \in \R^{n \times n};$
	\item solution of the linear system $f' \cdot w = v$ for a given $v \in \R^n,$
\end{enumerate}
yielding
$
w=(f')^{-1} \cdot v. 
$
The computational effort is typically dominated by the former, as the cost 
of evaluating $f$ often exceeds $\mathcal{O}(n^2).$ 
Symbolic (as well as structural) analysis of $f$ may yield options for 
avoiding the accumulation of the Jacobian. A notable gain in computational 
performance can be expected. 
Implicit Euler integration 
of initial value problems represents one prominent example, where $f$ amounts
to the sequence of Euler steps. Even the solution of the linear system becomes
obsolete in this case. The computational cost of the naive approach can thus be reduced by a factor of $\mathcal{O}(n^2).$


\begin{thebibliography}{10}

\bibitem{Berryman1992Toa}
A.~Berryman.
\newblock The origins and evolution of predator-prey theory.
\newblock {\em Ecology}, 73(5):1530--1535, 1992.

\bibitem{Butcher2016NMf}
C.~Butcher.
\newblock {\em Numerical Methods for Ordinary Differential Equations}.
\newblock Wiley \& Sons, 2016.

\bibitem{Chavent2010NLS}
G.~Chavent.
\newblock {\em Nonlinear Least Squares for Inverse Problems}.
\newblock Springer, 2010.

\bibitem{Deuflhard2004NMf}
P.~Deuflhard.
\newblock {\em Newton Methods for Nonlinear Problems. Affine Invariance and
  Adaptive Algorithms}, volume~35 of {\em Computational Mathematics}.
\newblock Springer International, 2004.

\bibitem{Gebremedhin2005WCI}
A.~Gebremedhin, F.~Manne, and A.~Pothen.
\newblock What color is your {J}acobian? {G}raph coloring for computing
  derivatives.
\newblock {\em SIAM Review}, 47(4):629--705, 2005.

\bibitem{Gilbert1992Ada}
J.~C. Gilbert.
\newblock Automatic differentiation and iterative processes.
\newblock {\em Optimization Methods and Software}, 1:13--21, 1992.

\bibitem{Griewank1992ALG}
A.~Griewank.
\newblock Achieving logarithmic growth of temporal and spatial complexity in
  reverse automatic differentiation.
\newblock {\em Optimization Methods and Software}, 1:35--54, 1992.

\bibitem{Griewank1996AAC}
A.~Griewank, D.~Juedes, and J.~Utke.
\newblock {Algorithm 755}: {ADOL-C}: A package for the automatic
  differentiation of algorithms written in {C\slash C++}.
\newblock {\em {ACM} Transactions on Mathematical Software}, 22(2):131--167,
  1996.

\bibitem{Griewank2008EDP}
A.~Griewank and A.~Walther.
\newblock {\em Evaluating Derivatives. Principles and Techniques of Algorithmic
  Differentiation, Second Edition}.
\newblock Number OT105 in Other Titles in Applied Mathematics. SIAM, 2008.

\bibitem{Hofbauer1998EGa}
J.~Hofbauer and K.~Sigmund.
\newblock {\em Evolutionary Games and Population Dynamics}.
\newblock Cambridge University Press, 1998.

\bibitem{Hogan2014FRM}
R.~Hogan.
\newblock Fast reverse-mode automatic differentiation using expression
  templates in {C++}.
\newblock {\em {ACM} Transactions on Mathematical Software}, 40(4):26:1--26:24,
  jun 2014.

\bibitem{Kelley2003SNE}
T.~Kelley.
\newblock {\em Solving Nonlinear Equations with Newton's Methods}.
\newblock SIAM, 2003.

\bibitem{lotz2011dco}
J.~Lotz, K.~Leppkes, and U.~Naumann.
\newblock dco/c++-derivative code by overloading in {C++}.
\newblock {\em Aachener Informatik Berichte (AIB-2011-06)}, 2011.

\bibitem{Naumann2008DRi}
U.~Naumann.
\newblock {DAG} reversal is {NP}-complete.
\newblock {\em Journal of Discrete Algorithms}, 7:402--410, 2009.

\bibitem{Naumann2024AME}
U.~Naumann.
\newblock A matrix-free exact {N}ewton method.
\newblock {\em SIAM Journal on Scientific Computing}, 46(3):A1423--A1440, 2024.

\bibitem{Naumann2015NLS}
U.~Naumann, J.~Lotz, K.~Leppkes, and M.~Towara.
\newblock Algorithmic differentiation of numerical methods: Tangent and adjoint
  solvers for parameterized systems of nonlinear equations.
\newblock {\em ACM Transactions on Mathematical Software}, 41:26, 2015.

\bibitem{sagebaum2019high}
M.~Sagebaum, T.~Albring, and N.~Gauger.
\newblock High-performance derivative computations using {CoDiPack}.
\newblock {\em ACM Transactions on Mathematical Software}, 45(4):1--26, 2019.

\end{thebibliography}
\end{document}